\documentclass[11pt,reqno]{amsart}
\usepackage[left=2.5cm,right=2.5cm,top=2.5cm, bottom=2.5cm]{geometry}
\usepackage{amsfonts,epsfig,fancyhdr}
\usepackage{amsmath,amssymb,amsthm, mathrsfs, fancyhdr,latexsym,graphicx,graphics}

\usepackage{pdfsync}
\usepackage{url,hyperref}
\usepackage{euscript}
\usepackage{wrapfig}
\usepackage{caption}
\usepackage{tcolorbox}
\usepackage{graphicx}
\numberwithin{equation}{section}

\subjclass[2010]{Primary ...; ...}

\numberwithin{equation}{section}

\newtheorem{theorem}{Theorem}[section]

\newtheorem{lemma}[theorem]{Lemma}
\newtheorem{rem}[theorem]{Remark}

\newtheorem{prop}[theorem]{Proposition}

\theoremstyle{definition}
\newtheorem{definition}{Definition}[section]

\usepackage{enumerate}

\title{F--manifolds, F$_\text{man}$--algebras and Poisson--algebra distributions }

\author[S. Castañeda-Montoya]{Santiago Castañeda-Montoya}
\address{FaMAF, Universidad Nacional de Córdoba,    ciudad universitaria, (5000) Córdoba, Argentina.}
\email{santiago.castaneda@mi.unc.edu.ar}

\author[A. Torres-Gomez]{Alexander Torres-Gomez}
\address{Instituto de Matemáticas, FCEyN, Universidad de Antioquia, 50010  Medellín, Colombia}

\email{galexander.torres@udea.edu.co}


\begin{document}

\begin{abstract} 
This paper investigates the geometric and algebraic interplay between F--manifolds and a newly defined class of structures termed F$_\text{man}$--algebras. We specialize our study to the category of F--Lie groups, characterized by a Lie group whose associated commutative and associative product of vector fields is left--invariant.  We construct a canonical connection on Lie groups uniquely determined by the F$_\text{man}$--algebraic data, and subsequently characterize its curvature tensor and holonomy Lie algebra. A central feature of our investigation is the introduction of the Poisson--algebra distribution, arising from a canonical Poisson subalgebra within the F$_\text{man}$--algebra. We establish the integrability of this distribution, which induces a foliation of the F--Lie group and facilitates a local splitting theorem. The theoretical framework is illustrated through an in--depth analysis of the Heisenberg Lie algebra.

\end{abstract}

\maketitle

\tableofcontents

%%%%%%%%%%%%%%%%%%%%%%%%%%%%%%%%%Sec
\section{Introduction}

The study of Frobenius manifolds was pioneered by Dubrovin \cite{Du}, providing a crucial geometrical perspective on the Witten--Dijkgraaf--Verlinde--Verlinde equation, which originated in two--dimensional topological field theories. Building on this foundation, Hertling and Manin \cite{HM} later introduced the concept of an F--manifold (originally called a weak Frobenius manifold). This geometric construction relaxes some of the initial conditions required for a Frobenius manifold, serving as the primary motivation for our current investigation. Notably, all Frobenius manifolds are inherently F--manifolds.\\

F--manifolds are ubiquitous, appearing in diverse mathematical context such as singularity theory, quantum cohomology, symplectic geometry, and the theory of integrable systems. For a comprehensive discussion and relevant references, see \cite{He, Ma}. This broad applicability has motivated the study of various associated geometrical and algebraic structures, including double extensions of flat pseudo--Riemannian F--Lie algebras \cite{TGV}, Riemannian and bi--flat F--manifolds \cite{ABLR}, biflat F--structures as differential bicomplexes \cite{AL}, F--algebroids \cite{CMTG}, F-algebra--Rinehart pairs \cite{CMGTG, CMLS}, F--manifolds via pre--Lie algebras \cite{Do}, relationship between Nijenhuis manifolds and F--manifolds \cite{AK}, and Hodge atoms arising from the spectral decomposition of an F--bundle \cite{KKPY}. Studying these structures aims to advance the general understanding of these geometric spaces. \\

Roughly speaking, an F--manifold consists of a smooth manifold $M$ whose $C^\infty(M)$-module of vector fields, $\mathfrak{X}(M)$, is endowed with a commutative and associative product $\circ$. This structure is subject to an integrability constraint known as the Hertling–Manin condition, which relates the Lie bracket and the commutative product through a specific cubic identity (see Definition \ref{DefFman}). The present work focuses on the setting where $M$ is a Lie group and both the vector fields and the commutative product are left--invariant; we refer to such a structure as an F--Lie group.  In this framework, we identify an underlying algebraic structure termed an F$_\text{man}$--algebra (see Definition \ref{DefFmanalg}). Notably, the F$_\text{man}$--algebraic axioms can be derived by relaxing the constraints of a standard Poisson algebra. Consequently, F$_\text{man}$--algebras admit a purely abstract algebraic treatment, independent of the encompassing manifold geometry.\\

On an F--Lie group, we define a canonical connection uniquely determined by the underlying F$_\text{man}$--algebra structure. By investigating the resulting geometric properties, specifically the curvature tensor and the associated holonomy Lie algebra, we establish a geometric characterization of F$_\text{man}$--algebras analogous to the treatment of Poisson algebras in \cite{BB}. Central to our analysis is the identification of a Poisson subalgebra within the F$_\text{man}$--algebra, which induces a Poisson--algebra distribution on the Lie group. We prove that this distribution is completely integrable, thereby inducing a foliation of the F--Lie group.\\

The remainder of this paper is structured as follows. Section \ref{SecFstruc} provides the formal foundations for F--manifolds and F$_\text{man}$--algebras. Section \ref{SecFLieG} is dedicated to the study of F--Lie groups, where we derive the canonical connection associated with the F--structure, compute its curvature tensor, and characterize its holonomy Lie algebra. Within this section, we also define the Poisson--algebra distribution, establish its integrability, and describe the resulting foliation of the F--Lie group. Furthermore, we provide a local splitting theorem for the F--Lie group. The paper concludes with a detailed application of this framework to the Heisenberg Lie algebra, illustrating the interplay between its algebraic and geometric properties.

%%%%%%%%%%%%%%%%%%%%%%%%Section
\section{F--structures}
\label{SecFstruc}

\subsection{F--manifolds}
The concept of an F--manifold was introduced and investigated by Hertling and Manin in \cite{HM}. Intutively, an F--manifold is a smooth manifold $M$ whose $C^\infty(M)$--module of vector fields, $\mathfrak{X}(M)$, is equipped with a commutative and associative bilinear product, denoted by $\circ$. This structure is required to satisfy a specific integrability constraint known as the Hertling--Manin condition, which involves a cubic compatibility relation between the Lie bracket and the commutative product $\circ$.\\

Let $M$ be a smooth manifold. The $C^\infty(M)$--module of smooth vector fields $\mathfrak{X}(M)$ is equipped with the following algebraic structures:

\begin{enumerate}
\item \textbf{Lie Bracket:} The standard Lie bracket $[\cdot, \cdot]: \mathfrak{X}(M) \times \mathfrak{X}(M) \to \mathfrak{X}(M)$, where $[X, Y]$ denotes the commutator of the vector fields $X$ and $Y$.

\item \textbf{Commutative Product}: An associative and commutative $C^\infty(M)$--bilinear product $\circ: \mathfrak{X}(M) \times \mathfrak{X}(M) \to \mathfrak{X}(M)$.
\end{enumerate}

For any triplet of vector fields $X, Y, Z \in \mathfrak{X}(M)$, we define the Leibnizator as the trilinear map $L: \mathfrak{X}(M) \times \mathfrak{X}(M) \times \mathfrak{X}(M) \to \mathfrak{X}(M)$ given by:
\begin{equation}\label{Leib}
L(X,Y,Z):=[X,Y \circ Z ]-[X,Y ]\circ Z - Y \circ [X, Z ] \, .
\end{equation}

\begin{rem}
The Leibnizator $L$ is $C^\infty(M)$--multilinear with respect to its second and third arguments. However, it fails to be $C^\infty(M)$--linear in its first argument, as the Lie bracket acts as a derivation rather than a $C^\infty(M)$--linear operator. Consequently, $L$ does not define a tensor field on $M$. Furthermore, the commutativity of the product $\circ$ implies that $L$ is symmetric in its second and third arguments, i.e., $L(X, Y, Z) = L(X, Z, Y)$ for all $X, Y, Z \in \mathfrak{X}(M)$.
\end{rem}

\begin{rem}\label{HMcon}
It is essential to emphasize that the Leibnizator $L$ satisfies the so-called Hertling--Manin condition for all vector fields $X, Y, Z, W \in \mathfrak X(M)$:
\begin{equation}\label{HMcond}
L(X \circ Y,Z,W)= L(X,Z,W) \circ Y + X \circ L( Y,Z,W) \, .
\end{equation}
This identity signifies that the Leibnizator acts as a derivation with respect to the commutative product $\circ$ in its first argument, providing a deep compatibility between the Lie and commutative algebraic structures on $\mathfrak X(M)$.
\end{rem}

\begin{definition}\label{DefFman}
 An F--manifold is a smooth manifold $M$ equipped with a $C^\infty(M)$--bilinear, associative and commutative product $\circ$ on the module of smooth vector fields $\mathfrak{X}(M)$, such that the associated Leibnizator $L$ satisfies the Hertling--Manin condition (Equation \ref{HMcon}).
\end{definition}

\begin{rem}
It is worth noting that while the original definition of an F--manifold introduced by Hertling and Manin does not require the commutative product to possess a unit vector field (see \cite[Definition 1]{HM}), some authors include the existence of a unit in their definition (see \cite[Definition 2.8]{He}). Requiring a unit imposes additional constraints on the structures defined over an $F$-manifold.
\end{rem}

\begin{definition}
An F--Poisson manifold is an F--manifold whose commutative product $\circ$ and Lie bracket $[\cdot, \cdot]$ are compatible in the sense that the Leibnizator vanishes identically ($L \equiv 0$).
\end{definition}

\begin{rem}
The commutative product $\circ$ on $\mathfrak{X}(M)$ admits a geometric interpretation as a symmetric $(1,2)$--tensor field $S \in \Gamma(T^{(1,2)}M)$. This tensor is defined by the assignment $S(X, Y) := X \circ Y$ for all $X, Y \in \mathfrak{X}(M)$, where its symmetry is a direct consequence of the commutativity of the product.
\end{rem}

The following lemma establishes a geometric characterization of the Leibnizator $L$ by expressing it in terms of the symmetric $(1,2)$--tensor field $S$ and the Lie derivative operator $\mathcal{L}$.

\begin{lemma} \label{HMwithS}
Let $S \in \Gamma(T^{(1,2)}M)$ be the symmetric tensor field associated with the commutative product $\circ$, defined by $S(Y, Z) = Y \circ Z$ for all $Y, Z \in \mathfrak{X}(M)$. Then:
\begin{enumerate}
\item The Leibnizator $L(X, Y, Z)$ coincides with the Lie derivative of the tensor field $S$ along the vector field $X$. Specifically, we have: $(\mathcal{L}_{X}S)(Y,Z)=L(X,Y,Z)$.

\item The Hertling–Manin condition (Equation \ref{HMcond}) is equivalent to the following tensor identity: $\mathcal L_{S(X,Y)} S= S(X,\mathcal L_Y S) + S(Y, \mathcal L_X S)$.
\end{enumerate}
\end{lemma}

\begin{proof}
Recall that the Lie derivative $\mathcal L_X$ acts as a derivation of the tensor algebra and commutes with all contractions. By evaluating the Lie derivative of the $(1,2)$--tensor field $S \in \Gamma(T^{(1,2)}M)$ on the pair of vector fields $(Y, Z)$, we obtain:

\begin{align*}
        (\mathcal{L}_X S)(Y,Z)&=\mathcal{L}_X(S(Y,Z))-S(\mathcal{L}_XY,Z)-S(Y,\mathcal{L}_X Z)\\
        & =[X, Y \circ Z] - [X,Y] \circ Z - Y \circ [X,Z]\\
        &=L(X,Y,Z). 
    \end{align*} 
 The second assertion is then a direct consequence of substituting the Lie derivative $\mathcal{L}_X S$ for $L(X, \cdot, \cdot)$ into the Hertling--Manin condition.
\end{proof}

An F--manifold can be equivalently defined as a pair $(M, S)$, consisting of a smooth manifold $M$ and a symmetric tensor field $S$ of type $(1, 2)$ on $M$, which satisfies the Hertling--Manin condition.

\begin{rem}\label{tensorHT}
The Hertling--Manin condition can be understood as the vanishing of the Hertling–Manin tensor field $N: \mathfrak X(M)^{\times 4} \to \mathfrak X(M)$, defined by:
\begin{align*}
    N(X,Y,Z,W) &= L(X \circ Y, Z, W) - L(X, Z, W) \circ Y - X \circ L(Y, Z, W)\\
    &=(\mathcal{L}_{S(X,Y)}S)(Z,W)- X\circ(\mathcal{L}_{Y}S )(Z,W)-  Y\circ(\mathcal{L}_{X}S )(Z,W).
\end{align*}
\end{rem}

%%%%%%%%%%%%%%%%%%%%%%%%Subsection
\subsection{F$_\text{man}$--algebras}
The concept of an F$_\text{man}$--algebra admits a dual characterization, serving both as a formal algebraic extension and a geometric byproduct. Specifically, it can be introduced as:
\begin{itemize}
\item A broadening of the class of Poisson algebras achieved by relaxing specific structural constraints.

\item The algebraic data induced on an F--manifold when the underlying manifold is a Lie group and the commutative and associative product is required to be left--invariant.
\end{itemize}

In this subsection, we provide a formal definition of F$_\text{man}$--algebras motivated by a systematic relaxation of the Poisson axioms. In Section \ref{SecFLieG}, we establish the geometric correspondence by deriving the F$_\text{man}$--algebra structure from the framework of F--Lie groups.\\

We first recall the definition of a Poisson algebra (see \cite[Definition 1.1]{LGPV}).

\begin{definition}
A Poisson algebra is a vector space $\mathscr P$ over the field $\mathbb R$, equipped with two bilinear operations: an associative and commutative product, $\circ: \mathscr P \times \mathscr P \to \mathscr P$, and a Lie bracket, $[\cdot, \cdot]: \mathscr P \times \mathscr P \to \mathscr P$, such that:
\begin{enumerate}
\item  $(\mathscr P, \circ)$ is a commutative and associative  algebra.

\item $(\mathscr P, [\; , \;] )$ is a Lie algebra.

\item The Lie bracket acts as a derivation of the associative product; specifically, the Leibniz rule holds for all $u, v, w \in \mathscr P$: $[u,v \circ  w ]=[u,v ]\circ w+ v \circ [u, w ]$.
\end{enumerate}
\end{definition}

We define the Leibnizator $L$ as the trilinear map given by:
\begin{equation*} 
L(u,v,w):=[u,v \circ w ]-[u,v ]\circ w - v \circ [u, w ] \, .
\end{equation*}

The Leibnizator serves to quantify the failure of the Lie bracket to act as a derivation of the associative product; specifically, $L$ vanishes identically if and only if the structure is a Poisson algebra.

Having introduced the Leibnizator, we are now prepared to state the definition of an F$_\text{man}$--algebra \cite[Definition 1]{CMGTG}.

\begin{definition}\label{DefFmanalg}
An F$_\text{man}$--algebra is a triple 
$(\mathscr F, \circ, [\; , \;])$ such that:
\begin{enumerate}
\item $(\mathscr F, \circ)$ is a commutative and associative algebra.

\item $(\mathscr F, [\; , \;] )$ is a Lie algebra.

\item The Leibnizator $L$ satisfies the following identity for all $u, v, w, y \in \mathscr F$:
\begin{equation}
L(u \circ v,w,y)= L(u,w,y) \circ v + u \circ L( v,w,y) \, .
\end{equation}
\end{enumerate}

This identity asserts that the Leibnizator $L$ acts as a derivation in its first argument with respect to the associative product $\circ$. This equation can also be called the Hertling--Manin condition.
\end{definition}

It follows directly from the definitions that a Poisson algebra is an F$_\text{man}$--algebra whose Leibnizator $L$ is identically zero.  In this sense, F$_\text{man}$--algebras may be regarded as almost Poisson algebras, representing a structural relaxation of the standard Poisson compatibility condition.

%%%%%%%%%%%%%%%%%%%Section
\section{F--Lie groups} \label{SecFLieG}

In this section, we specialize our analysis to F--manifolds $(G, \circ)$ where the underlying manifold $G$ possesses the additional structure of a simply connected Lie group and the commutative product $\circ$ is left--invariant. In this context, left--invariance stipulates that for any pair of left--invariant vector fields $X, Y \in \mathfrak{X}^+(G)$, their product $X \circ Y$ is likewise left--invariant. We formally define such a structure as an F--Lie group. The geometric and algebraic analysis in this setting is naturally concentrated on the space of left--invariant vector fields. This space constitutes a finite--dimensional real vector space which is endowed with the structure of an F$_\text{man}$--algebra. As is standard in Lie theory, there exists a canonical vector space isomorphism between $\mathfrak{X}^+(G)$ and the tangent space at the identity, $\mathfrak{g} \cong T_eG$. Consequently, the F$_\text{man}$--algebraic data on $\mathfrak{X}^+(G)$ induces a corresponding F$_\text{man}$--algebra structure on the Lie algebra $\mathfrak{g}$.\\

Given $u^+, v^+ \in \mathfrak{X}^+(G)$, let $l_g: G \to G$ denote the left translation by an element $g \in G$, defined by $l_g(h) = gh$. The left--invariance of the commutative and associative product $\circ$ is characterized by the condition that the push--forward $(l_g)_*$ satisfies $(l_g)_*(u^+ \circ v^+) = u^+ \circ v^+$. We identify each left--invariant vector field $u^+$ with its value at the identity $e \in G$, denoted by $u := u^+_e \in \mathfrak{g}$.

\begin{rem}
Recall that if $f \in C^\infty(G)$ is a smooth real--valued function on a connected Lie group $G$, and $v^+ \in \mathfrak{X}^+(G)$ is a non--trivial left--invariant vector field, then the product $f v^+$ is left--invariant if and only if $f$ is a constant function. Consequently, for any $u^+, v^+ \in \mathfrak{X}^+(G)$ and a constant function $f$, the Lie bracket satisfies:
$$
[u^+, f v^+]= (u^+f)\; v^+ + f [u^+, v^+]= f \; [u^+, v^+] 
$$
since $u^+ f$ vanishes for all constant functions. This confirms that the Lie bracket on $\mathfrak{X}^+(G)$ restricts to a well--defined $\mathbb{R}$--bilinear operation on the Lie algebra $\mathfrak{g}$.
\end{rem}

\begin{rem}
The Leibnizator $L$ (cf. Equation (\ref{Leib})), when restricted to the space of left--invariant vector fields $\mathfrak{X}^+(G)$ on an F--Lie group $G$, is explicitly given by:
$$
L(u^+,v^+,w^+)= [u^+,\, v^+ \circ w^+]
\;-\; [u^+,v^+] \circ w^+
\;-\; v^+ \circ [u^+,w^+]
$$
for any $u^+, v^+, w^+ \in \mathfrak{X}^+(G)$. Since the Lie bracket is inherently left--invariant and the commutative product $\circ$ is left--invariant by assumption, the Leibnizator $L$ constitutes a left--invariant trilinear map on $G$. Consequently, $L$ is uniquely determined by its evaluation at the identity element $e \in G$. This allows for a canonical restriction of the operation to the Lie algebra $\mathfrak{g} \cong T_e G$, yielding the algebraic Leibnizator:
$$L(u, v, w) = [u, v \circ w] - [u, v] \circ w - v \circ [u, w]$$
for all $u, v, w \in \mathfrak{g}$. In this context, the Lie bracket $[\cdot, \cdot]$ and the product $\circ$ on $\mathfrak{g}$ are the algebraic operations induced by the restriction of the corresponding invariant fields:
$$[u, v] := [u^+, v^+]|_e \quad \text{and} \quad u \circ v := (u^+ \circ v^+)|_e.$$ 
\end{rem}

It is well--established that there exists a bijective correspondence between the set of all left--invariant affine connections on a Lie group and the set of all bilinear products on its Lie algebra (see \cite[Chapter 6 \S 3]{Po} and \cite{BB}). The following definition is established within this setting.

\begin{definition}\label{LIconn}
Let $( \mathfrak X^+(G),[\,,\,], \circ)$ be the F$_\text{man}$--algebra of the F--Lie group $(G, \circ)$. We define a left--invariant affine connection $\nabla$ on $G$, which we call the canonical F--connection, by the following formula:
$$
\nabla_{u^+} v^+ = \frac{1}{2} [u^+,v^+] + u^+ \circ v^+,
$$
for all left--invariant vector fields $u^+,v^+$.
\end{definition}

We recall that the torsion tensor $\mathscr{T}$ of an affine connection $\overline{\nabla}$ is defined for vector fields $X$ and $Y$ by: $\mathscr T(X,Y)=\overline \nabla_X Y-\overline \nabla_Y X-[X,Y]$. We now apply this to the canonical F--connection $\nabla$ (Definition \ref{LIconn}) for left--invariant vector fields $u^+$ and $v^+$:
$$
\mathscr T(u^+,v^+)=u^+\circ v^+ - v^+\circ u^+ \, .
$$
Using  the commutativity of the F$_{\text{man}}$--product, $v^+ \circ u^+ = u^+ \circ v^+$, the expression simplifies to $\mathscr T(u^+,v^+)=0$. Therefore, the preceding calculation establishes the following lemma.
\begin{lemma}\label{TorFreeConn}
The canonical  F--connection $
\nabla_{u^+} v^+ = \frac{1}{2} [u^+,v^+] + u^+ \circ v^+$  is torsion--free.
\end{lemma}

%%%%%%%%%%%%%%%%%%%%%%%%%%%Subsec
\subsection{Curvature of the Affine Connection}
We begin by recalling the definition of the curvature tensor $\overline R$ for an affine connection $\overline \nabla$, acting on vector fields $X, Y, Z$: 
$$
\overline R(X,Y)Z := \left[ \overline \nabla_X, \overline \nabla_Y \right] Z - \overline \nabla_{[X,Y]} Z \, .$$

We compute the curvature $R(u^+, v^+) w^+$ for the canonical F--connection $\nabla$ (Definition \ref{LIconn}) acting on left--invariant vector fields $u^+, v^+, w^+$ on the F--Lie group $G$. Applying the definition $\nabla_{u^+} v^+ = \tfrac{1}{2} [u^+,v^+] + u^+ \circ v^+$ repeatedly:
\begin{align*}
\nabla_{u^+} \nabla_{v^+} w^+ &= \frac12\left( \frac12 [u^+,[v^+, w^+]] + u^+ \circ [v^+, w^+] \right)
  + \frac12 [u^+, v^+ \circ w^+] + u^+ \circ (v^+ \circ w^+) \, , \\
\nabla_{v^+} \nabla_{u^+} w^+ &= \frac12\left( \frac12 [v^+,[u^+ , w^+ ]] + v^+ \circ [u^+, w^+] \right)
  + \frac12 [v^+ , u^+ \circ w^+] + v^+ \circ (u^+ \circ w^+) \, ,\\
\nabla_{[u^+, v^+]} w^+  &= \frac12 [[u^+,v^+ ], w^+] + [u^+, v^+] \circ w^+ \, .
\end{align*}
Substituting these into the curvature definition, $R(u^+, v^+) w^+ = [\nabla_{u^+}, \nabla_{v^+}] w^+  - \nabla_{[u^+, v^+]} w^+$, and using the F$_{\text{man}}$--algebra properties  (specifically, the associativity and commutativity of $\circ$, which ensures $u^+ \circ (v^+ \circ w^+) - v^+ \circ (u^+ \circ w^+) = 0$), we obtain:
\begin{align}\label{Rie}
    R(u^+,v^+)& w^+ =  \tfrac{1}{4}\left( [u^+,[v^+, w^+]] - [v^+,[u^+, w^+]] \right) - \tfrac{1}{2} [[u^+,v^+], w^+] \notag \\
              +\tfrac12 &[u^+, v^+ \circ w^+] - \tfrac12 [v^+, u^+ \circ w^+] - ([u^+,v^+] \circ w^+) +\tfrac12 (u^+ \circ [v^+, w^+] - v^+ \circ [u^+, w^+]) \, .
\end{align}

We observe two structural simplifications: 
\begin{enumerate}[(i)]
\item The second line can be expressed using the Leibnizator $L$ as $\tfrac12 \left(L(u^+,v^+,w^+) - L(v^+,u^+,w^+)\right)$.

\item The first line, which involves nested Lie brackets, simplifies using the Jacobi identity to $-\tfrac{1}{4} [[u^+, v^+], w^+]$.
\end{enumerate}

The preceding computations regarding the canonical F--connection are summarized in the following theorem, which expresses the curvature tensor in terms of the Lie algebra structure and the Leibnizator.

\begin{theorem}\label{curvaturethm}
Let $(\mathfrak X^+(M), [\,,\,], \circ)$ be the F$_{\text{man}}$--algebra of the F--Lie group $(G, \circ)$, endowed with the left--invariant connection $\nabla$ defined in Definition \ref{LIconn}. Then, the curvature tensor $R$ of $\nabla$, acting on the left--invariant vector fields $u^+, v^+, w^+$, is given by:
$$R(u^+,v^+)w^+ =- \tfrac14 [[u^+,v^+],w^+]+\tfrac12\left(L(u^+,v^+, w^+) - L(v^+,u^+,w^+)\right) \, . $$
\end{theorem}

\begin{rem}

Theorem \ref{curvaturethm} demonstrates that the curvature tensor $R$ of the canonical F--connection $\nabla$ splits into two geometrically distinct contributions, which are expressed solely in terms of the Lie bracket and the Leibnizator.

\begin{enumerate}[(i)]
\item The first term,
$
-\tfrac14 [[u^+,v^+],w^+]
$, 
depends only on the Lie bracket $[\,,\,]$. This term is precisely the curvature of the canonical torsion--free Cartan connection on the Lie group $G$. Hence, it represents the intrinsic curvature associated solely with the underlying Lie group structure of $G$.

\item The remaining term,
$
\tfrac12\big(L(u^+,v^+,w^+)-L(v^+,u^+,w^+)\big)
$,  arises entirely from the Leibnizator $L$. This part measures the deviation from the condition that the Lie bracket acts as a derivation of the F$_{\text{man}}$--product $\circ$. Geometrically, this contribution quantifies how the symmetric tensor $S(u^+, v^+) = u^+ \circ v^+$ (which defines the F$_{\text{man}}$--algebra and is external to the standard Lie group structure) bends or distorts the connection $\nabla$.
\end{enumerate}
\end{rem}

\begin{rem}
The analysis of the curvature formula provided in Theorem \ref{curvaturethm} leads to two important special cases by examining the contributions from the Lie bracket and the Leibnizator.
\begin{enumerate}[(i)]
\item If the Leibnizator vanishes identically, $L \equiv 0$, then the F$_{\text{man}}$--algebra structure reduces to a Poisson algebra structure. In this case, the curvature simplifies to:
$$R(u^+,v^+) w^+ = -\frac14 [[u^+,v^+],w^+] \, .$$
This expression is precisely the bi-invariant curvature determined solely by the Lie bracket on $\mathfrak X^+(G)$. This result aligns with the analysis of special connections studied in the context of Poisson algebras in \cite{BB}.

\item Consider the case where $\mathfrak X^+(G)$ is a 2--step nilpotent Lie algebra. By definition, this means that the double Lie bracket vanishes: $[[u^+, v^+], w^+] = 0$ for all $u^+, v^+, w^+ \in \mathfrak X^+(G)$. Geometrically, this condition implies that the contribution to the curvature originating from the Lie bracket structure vanishes, effectively removing the intrinsic Lie group curvature term. In this specific case, the curvature tensor of $\nabla$ reduces entirely to the second term:
$$
R(u^+,v^+) w^+ = \tfrac12\left(L(u^+,v^+,w^+)-L(v^+,u^+,w^+)\right) \, .
$$
Thus, for F--Lie groups whose Lie algebra is 2--step nilpotent, the geometry of the canonical connection $\nabla$ is governed solely by the interaction between the commutative product $\circ$ and the underlying nilpotent structure, as captured by the Leibnizator $L$. The specialized case of 2--step nilpotent Lie groups where $L \equiv 0$ (i.e., the Poisson algebra case) is treated in detail in \cite{BB}. 
\end{enumerate}
\end{rem}

%%%%%%%%%%%%%%%%%%%%%%%%%%%%%%%%%%%%%%%%%%subsec
\subsection{Holonomy Algebra}
We begin by recalling the Ambrose-Singer Theorem \cite{AS}, which provides a method for determining the Lie algebra of the holonomy group of a linear connection. Specifically, it characterizes the Lie algebra $\mathfrak{hol}_p(\overline{\nabla})$ of the restricted holonomy group $\mathrm{Hol}^0_p(\overline \nabla)$ (based at $p \in M$) in terms of the curvature endomorphisms $R(X, Y)$ and all of their successive covariant derivatives $\left((\nabla_{Z_1} \cdots \nabla_{Z_k} R)(X, Y)\right)$ evaluated at $p$, for all vector fields $X, Y, Z_1, \dots, Z_k$.

The purpose of this section is to establish key properties concerning the holonomy of the canonical connection $\nabla_{u^+} v^+ = \tfrac{1}{2} [u^+, v^+] + u^+ \circ v^+$ when the pair of products $([\,,\,], \circ)$ defines either an F$_{\text{man}}$--algebra structure or a Poisson algebra structure. We denote the canonical connection associated with a Poisson algebra structure as $\nabla^{\text{Pois}}$. The curvature tensor associated with $\nabla^{\text{Pois}}$ is denoted by $R^{\text{Pois}}$.\\

On the F--Lie group $G$, we consider the left--invariant connection defined by the assignment: 
$$\nabla_{u^+} v^+ = \tfrac{1}{2}[u^+, v^+] + u^+ \circ v^+\, .$$ 
For each element $u \in \mathfrak{g}$, we define an associated endomorphism $A_u \in \operatorname{End}(\mathfrak{g})$ given by:
$$A_u := \tfrac{1}{2}\operatorname{ad}_u + S_u$$
where $\operatorname{ad}_u(v) = [u, v]$ denotes the adjoint representation and $S_u(v) = u \circ v$ is the linear operator induced by the F$_\text{man}$--product. Under the canonical isomorphism between the space of left--invariant vector fields $\mathfrak{X}^+(G)$ and the Lie algebra $\mathfrak{g}$, the covariant derivative $\nabla_{u^+}$ corresponds precisely to the algebraic endomorphism $A_u$. Consequently, the geometric properties of the connection $\nabla$ may be investigated through the algebraic properties of the operators $A_u$ acting on $\mathfrak{g}$.\\

We list several elementary and structural facts in the following remark, which will be instrumental in the subsequent description of the holonomy Lie algebra of the left--invariant connection $\nabla$.

\begin{rem}\label{EndConnCurv}
The curvature endomorphism $R(u^+, v^+)$ evaluated at the identity $e \in G$, denoted $R(u, v) \in \text{End}(\mathfrak{g})$, can be decomposed into two distinct endomorphisms:
$$
R(u,v) \;=\; -\tfrac14\operatorname{ad}_{[u,v]} +\tfrac{1}{2}\big(L(u,v)-L(v,u)\big) \, .
$$
Here, $L(u, v) \in \text{End}(\mathfrak{g})$ denotes the Leibnizator endomorphism (acting on the third argument): $L(u,v)=[\operatorname{ad}_u, S_v]-S_{[u,v]}$.

With the previous definitions and assumptions, the following results are achieved:
\begin{enumerate}[(i)]
\item  For any $v \in \mathfrak{g}$, considering the endomorphism $\operatorname{ad}_v \in \mathrm{End}(\mathfrak{g})$, its covariant derivative with respect to $\nabla$ along $u \in \mathfrak{g}$ is given by: $\nabla_u (\operatorname{ad}_v) = [A_u, \operatorname{ad}_v]$.

\item Using $A_v = \tfrac{1}{2}\operatorname{ad}_v + S_v$, the commutator $[\operatorname{ad}_u, A_v]$ is calculated as:
\begin{align*}
[\operatorname{ad}_u,A_v]
=& \tfrac12[\operatorname{ad}_u,\operatorname{ad}_v] \;+\; [\operatorname{ad}_u,S_v]=\tfrac12\operatorname{ad}_{[u,v]}+ [\operatorname{ad}_u,S_v] =\tfrac12\operatorname{ad}_{[u,v]}+S_{[u,v]}+L(u,v) \notag \\ 
=&A_{[u,v]}+L(u,v)\, .
\end{align*}

\item  By the definition of the curvature tensor $R(u^+, v^+) = \left[ \nabla_{u^+}, \nabla_{v^+} \right] - \nabla_{[u^+, v^+]}$, the commutator of the endomorphisms $A_u$ and $A_v$ (associated to the connection $\nabla$) is:
$$[A_u,A_v]=A_{[u,v]}-\tfrac{1}{4}\operatorname{ad}_{[u,v]}+\tfrac{1}{2}\left(L(u,v)-L(v,u)\right) \, .$$
\end{enumerate}
\end{rem}

The Ambrose-Singer Theorem applied to a Lie group with a left--invariant connection simplifies considerably, providing an effective characterization of the holonomy algebra in terms of the Lie algebra structure (see \cite[Theorem 2.3]{AGT} and \cite[Definition 2.3]{Al}).

\begin{theorem}\label{AS}
Let $\nabla$ be a left--invariant linear connection on the Lie group $G$, and let $\mathfrak{g}\cong \mathfrak X^+(G)$ be the Lie algebra of $G$. Then, the holonomy algebra $\mathfrak{hol}(\nabla)$, based at the identity element $e \in G$, is the smallest Lie subalgebra of $\operatorname{End}(\mathfrak{g})$ that satisfies two conditions: 
\begin{enumerate}
\item It contains all the curvature endomorphisms $R(v, w) \in \operatorname{End}(\mathfrak{g})$ for any $v, w \in \mathfrak{g}$.

\item It is closed under commutation with the endomorphism $A_u \in \operatorname{End}(\mathfrak{g})$, for any $u \in \mathfrak{g}$.
\end{enumerate}
\end{theorem}

We emphasize that, for a left--invariant connection, the covariant derivative of the curvature $(\nabla_u R)(v, w)$ and the commutator $[A_u, R(v, w)]$ are equivalent, differing only by curvature terms.\\

In the event that the Leibnizator vanishes identically, $L \equiv 0$, the structure $(\mathfrak{g}, [\,,\,], \circ)$ is a Poisson algebra. The curvature tensor of the resulting connection, denoted $\nabla^{\text{Pois}}$, simplifies to: 
$$R^{\text{Pois}}(u,v)=-\tfrac14\operatorname{ad}_{[u,v]} \, .$$

For this specific connection $\nabla^{\text{Pois}}$, the holonomy algebra $\mathfrak{hol}(\nabla^{\text{Pois}})$ is determined by the adjoint and the $A$ endomorphism of the Lie algebra. Specifically, it is given by:

$$\mathfrak{hol}(\nabla^{\text{Pois}}) = \operatorname{ad}_{[\mathfrak g, \mathfrak{g}]} + A_{[\mathfrak{g}, [\mathfrak{g}, \mathfrak{g}]]} \, ,$$ 

see \cite[Lemma 2.2]{BB}. In the general case of an F--Lie group, the relationship between the holonomy of the canonical F--connection $\nabla$ and the Poisson case holonomy $\nabla^{\text{Pois}}$ is described by the following proposition.

\begin{prop}\label{thm:holonomy}
Let $G$ be an F--Lie group endowed with the canonical connection $\nabla$ (Definition \ref{LIconn}). Then, the holonomy algebra of the associated Poisson connection $\nabla^{\text{Pois}}$ satisfies the following inclusion:
$$
\mathfrak {hol}(\nabla^{\text{Pois}})\subseteq \mathfrak{hol}(\nabla) \; .
$$
\end{prop}

\begin{proof}
The inclusion relation follows directly from the structural identities established in Remark \ref{EndConnCurv} together with the Ambrose--Singer theorem (Theorem \ref{AS}).\\
Furthermore, recall that by the Ambrose–Singer theorem, $\mathfrak{hol}(\nabla)$ is the smallest subalgebra of $\text{End}(\mathfrak g)$ satisfying the condition $(1)$ and $(2)$ of Theorem \ref{AS}. Applying Remark  \ref{EndConnCurv} yields the inclusion
$$
\mathfrak{hol}(\nabla) \subseteq \operatorname{ad}_{[\mathfrak g, \mathfrak g]}+A_{[\mathfrak g,[\mathfrak g,\mathfrak g ]]}+ L(\mathfrak g, \mathfrak g)
  \, . 
$$ 
In particular, when the Leibnizator vanishes (which corresponds to the Poisson algebra case), the holonomy algebra reduces to:
$$\mathfrak{hol}(\nabla^{\text{Poiss}}) = \operatorname{ad}_{[\mathfrak{g}, \mathfrak{g}]} + A_{[\mathfrak{g}, [\mathfrak{g}, \mathfrak{g}]]}.$$
\end{proof}

%%%%%%%%%%%%%%%%%%%%%%%%Subsection

\subsection{Poisson--Algebra Distribution}
A fundamental concept for our subsequent discussion is the notion of a distribution on a manifold. We now recall its formal definition and key properties (see \cite[Section 3.7]{BL}, \cite{La}).

\begin{definition}
Let $M$ be a smooth manifold and $TM \to M$ its tangent bundle. A distribution $\mathcal{D}$ in $TM$ is a mapping that assigns to each point $p \in M$ a vector subspace $\mathcal{D}_p \subseteq T_p M$. Furthermore:
\begin{enumerate}[(i)]
\item The rank of the distribution $\mathcal D$ at $p \in M$ is defined as the dimension of $\mathcal D_p$. 

\item The distribution $\mathcal D$ is said to be regular at $p$ if there exist an open set $U \subseteq M$ containing $p$ such that the rank of $\mathcal D$ is constant for all $q \in U$. 

\item The distribution $\mathcal D$ in $TM$ is smooth if for any $p\in M$
and $v \in \mathcal D_p$, there exist a smooth local vector field $X$ defined on an open set $U$ of $p$ such that $X_q  \in \mathcal D_q$ for all $q \in U$ and $X_p = v$.

\item A local
integral manifold through $p$ for $\mathcal D$ is a connected immersed submanifold $N \subseteq M$, $p \in N$,
such that $T_q N = \mathcal D_q$ for every $q \in N$.
\end{enumerate}
\end{definition}

\begin{rem}\label{regular-smooth} Let $(G, S)$ be an F--Lie group.
Consider the map $\mathcal L_{\bullet} S: \mathfrak X^+(G) \to \Gamma(T^{(1,2)}(TG))$, $u^+ \mapsto \mathcal L_{u^+} S$.  The kernel of this map, $\ker (\mathcal L_{\bullet} S)$, naturally defines a distribution on $G$. Defining a distribution via the kernel of a vector bundle map is a common technique in symplectic and Poisson geometry. For instance, in the context of presymplectic manifolds, a null distribution is defined as the kernel of the presymplectic 2-form (see \cite[Section 2.2.1]{Bu}).
\end{rem}

We now introduce the Poisson--algebra distribution on an F--Lie group and provide its geometric characterization in terms of the vanishing of the Lie derivative of the tensor field $S$.

\begin{definition}\label{nulidad} 
Let $(G, S)$ be an F--Lie group. We define the Poisson--algebra distribution, denoted by $\mathcal{D}_{\text{Pois}}$, as the kernel of the map $\mathcal{L}_{\bullet} S: \mathfrak{X}^+(G) \to \Gamma(T^{(1,2)}(T G))$, $u^+ \mapsto \mathcal L_{u^+} S$. Formally, this subbundle is defined by:
$$
\mathcal D_{\text{Pois}} := \{ u^+\in\mathfrak X^+(G) ;\; \mathcal L_{u^+} S= 0 \}.
$$
Equivalently, $\mathcal{D}_{\text{Pois}}$ can be characterized as the set of left--invariant vector fields that reside in the center of the Leibnizator $L$:
\begin{equation*}
\mathcal D_{\text{Pois}} := \{ u^+\in \mathfrak X^+(G) ;\;
L(u^+,v^+,w^+)= 0, \; \forall v^+, w^+ \in \mathfrak X^+(G)  \}.
\end{equation*}
\end{definition}

\begin{prop}\label{subalgebra}
Let $(G, \circ)$ be an F--Lie group and let $\mathcal{D}_{\text{Pois}}$ denote its associated Poisson--algebra distribution. Then, $\mathcal{D}_{\text{Pois}}$ is closed under the commutative product $\circ$; specifically, for any pair of left--invariant vector fields $u^+, v^+ \in \mathcal{D}_{\text{Pois}}$, their product $u^+ \circ v^+$ also resides in $\mathcal{D}_{\text{Pois}}$.
\end{prop}

\begin{proof}
The result is a direct consequence of the Hertling--Manin condition, which stipulates that for any F--manifold, the identity
$$L(u^+ \circ v^+, w^+, y^+) = L(u^+, w^+, y^+) \circ v^+ + u^+ \circ L(v^+, w^+, y^+)$$
holds for all $u^+, v^+, w^+, y^+ \in \mathfrak{X}^+(G)$. Let $u^+, v^+ \in \mathcal{D}_{\text{Pois}}$. By the definition of the Poisson--algebra distribution, the Leibnizator vanishes identically when its first argument is an element of $\mathcal{D}_{\text{Pois}}$, implying that $L(u^+, w^+, y^+) = 0$ and $L(v^+, w^+, y^+) = 0$ for all $w^+, y^+ \in \mathfrak{X}^+(G)$. Substituting these values into the Hertling--Manin identity, it follows that $L(u^+ \circ v^+, w^+, y^+) = 0$. Thus, $u^+ \circ v^+ \in \mathcal{D}_{\text{Pois}}$, confirming that the distribution is stable under the commutative product $\circ$.
\end{proof}

\begin{prop}\label{involutive}
Let $(G, S)$ be an F--Lie group and let $\mathcal{D}_{\text{Pois}}$ denote its associated Poisson--algebra distribution. Then, $\mathcal{D}_{\text{Pois}}$ is an involutive distribution; specifically, the space of left--invariant vector fields $\mathcal{D}_{\text{Pois}}$ is a Lie subalgebra of the Lie algebra $\mathfrak X^+(G)$.
\end{prop}

\begin{proof}
Let $u^+, v^+ \in \mathcal{D}_{\text{Pois}}$ be left--invariant vector fields. By the definition of the Poisson--algebra distribution (cf. Definition \ref{nulidad}), these fields satisfy the vanishing condition $\mathcal{L}_{u^+} S = 0$ and $\mathcal{L}_{v^+} S = 0$. We appeal to the fundamental identity relating the Lie bracket of vector fields to the commutator of their corresponding Lie derivatives
$\mathcal{L}_{[u^+, v^+]} = [\mathcal{L}_{u^+}, \mathcal{L}_{v^+}]$ (see  \cite[Proposition 3.4]{KN}). Applying this operator identity to the $(1,2)$--tensor field $S$, we obtain:
$$
\mathcal L_{[u^+,v^+]} S = [\mathcal L_{u^+},\mathcal L_{v^+}]S = \mathcal L_{u^+}(\mathcal L_{v^+} S)-\mathcal L_{v^+}(\mathcal L_{u^+} S)=0.
$$
Since the Lie derivative of $S$ along the bracket $[u^+, v^+]$ vanishes identically, the vector field $[u^+, v^+]$ resides in the kernel of the map $\mathcal{L}_{\bullet} S$. Consequently, $[u^+, v^+] \in \mathcal{D}_{\text{Pois}}$, which establishes that $\mathcal{D}_{\text{Pois}}$ is an involutive distribution.
\end{proof}

Recall that a foliation on a smooth manifold $M$ is defined as a partition of $M$ into a collection of connected, disjoint, immersed submanifolds, termed the leaves of the foliation \cite{La}. In light of Proposition \ref{involutive}, the Poisson--algebra distribution $\mathcal{D}_{\text{Pois}}$ is involutive; consequently, the Frobenius Theorem ensures its complete integrability. Our investigation focuses on the resulting foliation, where the leaves are the maximal integral manifolds of $\mathcal{D}_{\text{Pois}}$.

\begin{theorem}\label{thm:splitting}
 Let $(G, S)$ be an F--Lie group. We define the Poisson subspace at the identity by:
 $$(\mathcal{D}_{\text{Pois}})_e := \{ u \in \mathfrak{g}; \;  \mathcal{L}_{u^+} S = 0 \}$$
 where $u^+$ denotes the unique left--invariant vector field such that $u^+_e = u$. Then, the following assertions hold:
\begin{enumerate}
\item $(\mathcal{D}_{\text{Pois}})_e$ is a Lie subalgebra of $\mathfrak{g}$ whose integration yields a unique connected immersed Lie subgroup $H \subseteq G$. The leaves of the Poisson--algebra distribution are precisely the $H$--orbits $\{gH; \; g \in G\}$, each of which constitutes an immersed connected submanifold of $G$.

\item Let $r_h: G \to G$ denote the right translation by $h \in G$, defined by $r_h(g) = gh$. For any $u \in (\mathcal{D}_{\text{Pois}})_e$, the flow $\Phi^{u^+}_t: G \to G$ of the left--invariant vector field $u^+$, given by $\Phi^{u^+}_t(g) = r_{\exp(tu)}(g)$, is an automorphism of the tensor field $S$; that is, $(\Phi^{u^+}_t)^*S = S$. Consequently, the right action of $H$ on $G$ preserves $S$. In particular, $S$ is invariant along each $H$--orbit in the sense that $(r_h)^*S = S$ for all $h \in H$.

\item Furthermore, if the connected Lie subgroup $H$ is closed in $G$, then the canonical projection $\pi: G \to G/H$ constitutes a principal $H$--bundle such that the tensor field $S$ is constant along the fibers of $\pi$.
\end{enumerate}
\end{theorem}

\begin{proof}
\begin{enumerate}
\item Given the canonical vector space isomorphism $\mathfrak{X}^+(G) \cong \mathfrak{g}$ and the result of Proposition \ref{involutive}, it follows that $(\mathcal{D}_{\text{Pois}})_e$ is a Lie subalgebra of $\mathfrak{g}$. By the  the global version of Lie’s third theorem, there exists a unique connected immersed Lie subgroup $H \subseteq G$ whose Lie algebra is $(\mathcal{D}_{\text{Pois}})_e$. Since $\mathcal{D}_{\text{Pois}}$ is a left--invariant distribution of constant rank generated by a Lie subalgebra of $\mathfrak{g}$, it is completely integrable. The maximal integral manifolds are precisely the $H$--orbits $gH$, which constitute immersed connected submanifolds of $G$.

\item Fix $u \in (\mathcal{D}_{\text{Pois}})_e$. The flow $\Phi^{u^+}_t$ of the corresponding left--invariant vector field $u^+$ is generated by right multiplication: $\Phi^{u^+}_t = r_{\exp(tu)}$. By the definition of the Lie derivative of a tensor field, we have:
$$\frac{d}{dt}\Big|_{t=0} (r_{\exp(tu)})^*S = \mathcal{L}_{u^+}S = 0 \; .$$
This implies that the pullback $(r_{\exp(tu)})^*S = S$ for all $t \in \mathbb{R}$. Since $H$ is connected, it is generated by the image of the exponential map; thus, any $h \in H$ can be expressed as a finite product of elements of the form $\exp(u_i)$. It follows that $(r_h)^*S = S$ for all $h \in H$. Consequently, the right action of $H$ on $G$ preserves $S$, rendering $S$ invariant along each $H$--orbit.

\item Assume that the connected Lie subgroup $H \subseteq G$ integrating $(\mathcal{D}_{\text{Pois}})_e$ is closed in $G$. Standard results in Lie theory ensure that the right action of $H$ on $G$ is both free and proper. Furthermore, the quotient space $G/H$ admits a unique smooth manifold structure such that the canonical projection $\pi: G \to G/H$ is a smooth submersion \cite[Theorem 21.17]{Lee}.\\

Because $\pi$ is a smooth submersion, for any $gH \in G/H$, there exists a local smooth section $\sigma: U \to G$ defined on an open neighborhood $U \subseteq G/H$ such that $\pi \circ \sigma = \text{id}_U$ \cite[Theorem 4.26]{Lee}. This section induces a local trivialization of G, $ U \times H \cong \pi^{-1}(U)$, via the diffeomorphism $\phi: (u, h) \mapsto \sigma(u)h$. It follows that the fibers of $\pi$ are smooth submanifolds of $G$, each diffeomorphic to the Lie subgroup $H$. Therefore, $\pi: G \to G/H$ has the structure of a principal $H$--bundle. \\

From the results established in Items (1) and (2), we have already shown that the tensor field $S$ is invariant under the right $H$--action, i.e., $(r_h)^*S = S$ for all $h \in H$. Furthermore, the tangent space to the fibers of $\pi$ at any point $g \in G$ (the vertical distribution $\mathcal V$ of the bundle) is given by:
$$\mathcal V_g=\ker(d\pi_g) = T_g(gH) = (l_g)_* (\mathcal{D}_{\text{Pois}})_e \;.$$
By the construction of the Poisson--algebra distribution, the tensor field $S$ is constant along the fibers of $\pi$.

\end{enumerate}
\end{proof}

\begin{rem}\label{Foliation}
The local geometry of an F--Lie group $(G, S)$ in the neighborhood of a point $p \in G$ is governed by the structural decomposition provided in Theorem \ref{thm:splitting}-(3). Specifically, for a sufficiently small neighborhood $V \subseteq G$ of $p$, we have a local diffeomorphism $V \cong U \times H$. Under this decomposition, $H$ is the connected closed subgroup integrating the Poisson subalgebra $(\mathcal{D}_{\text{Pois}})_e$, representing the leaf of the foliation passing through $p$, while $U \subseteq G/H$ serves as a transverse parameter space. 

Crucially, because the tensor $S$ is invariant under the right $H$--action and its Lie derivative vanishes along the vertical distribution $\mathcal{V} \cong \mathfrak{h}$, the structural data of the F--manifold is locally constant along the leaves. Consequently, the entire functional variation of the commutative product $\circ$ is concentrated in the transverse directions (modeled on $U$), effectively reducing the study of the F--manifold structure to the geometry of the quotient manifold.
\end{rem}

The following theorem connects the Poisson--algebra distribution $\mathcal{D}_{\text{Pois}}$ directly to the local symmetries of the canonical F--connection $\nabla$ on the F--Lie group $G$.

\begin{theorem}\label{killing-affine}
Let $G$ be the simply connected F--Lie group endowed with the canonical connection $\nabla$ defined by $\nabla_{v^+} w^+ = \tfrac{1}{2} [v^+, w^+] + S(v^+, w^+)$. Let $\mathcal{D}_{\text{Pois}}$ be the Poisson--algebra distribution on $G$. For every left--invariant vector field $u^+$ on $G$ and all left--invariant vector fields $v^+, w^+$ on $G$, the following identity holds:
$$(\mathcal L_{u^+}\nabla)(v^+,w^+)=(\mathcal L_{u^+} S)(v^+,w^+)=L(u^+,v^+,w^+) \, .
$$
In particular, for every point $g \in G$, the fiber of the Poisson--algebra distribution at $g$ is characterized geometrically as the set of vectors that generate local infinitesimal affine transformations.

Equivalently, the sections of $\mathcal{D}_{\text{Pois}}$ are precisely the vector fields whose flows preserve the affine structure defined by the connection $\nabla$, i.e., left--invariant vector fields $u^+$ such that $\mathcal{L}_{u^+} \nabla = 0$.
\end{theorem}

\begin{proof}
The Lie derivative of an affine connection $\nabla$ with respect to a vector field $u^+$ is defined by:
$$
(\mathcal L_{u^+}\nabla)(v^+,w^+)
  = \mathcal L_{u^+}(\nabla_{v^+} w^+)
    - \nabla_{\mathcal L_{u^+} v^+} w^+
    - \nabla_{v^+}(\mathcal L_{u^+} w^+).
$$
Substituting the definition of the canonical F--connection $\nabla_{v^+} w^+ = \tfrac{1}{2} [v^+, w^+] + S(v^+, w^+)$ into this formula:
$$
(\mathcal L_{u^+} \nabla)(v^+,w^+)
  = \tfrac12\left(\mathcal L_{u^+}[v^+,w^+]
     - [\mathcal L_{u^+} v^+,w^+]
     - [v^+,\mathcal L_{u^+} w^+]\right)
    + (\mathcal L_{u^+} S)(v^+,w^+).
$$
The first group of terms, multiplied by $\tfrac{1}{2}$, vanishes identically due to the Jacobi identity. Thus, we are left with the identity: $
(\mathcal L_{u^+} \nabla)(v^+,w^+) = (\mathcal L_{u^+} S)(v^+,w^+)
$.

By Lemma \ref{HMwithS}, the Lie derivative of the F--manifold tensor $S$ is related to $L$ by $\mathcal L_{u^+} S=L(u^+,\cdot,\cdot)$. Combining this with the established identity, we have:
$$
\mathcal L_{u^+}\nabla=\mathcal L_{u^+} S=L(u^+,\cdot,\cdot) \, .
$$
The left--invariant vector field $u^+$ belongs to the Poisson--algebra distribution $\mathcal{D}_{\text{Pois}}$ if and only if $\mathcal{L}_{u^+} S = 0$ (Definition
\ref{nulidad}). Therefore, $u^+ \in \mathcal{D}_{\text{Pois}}$ if and only if $u^+$ is a local infinitesimal affine transformation (or symmetry) of the F--Lie group.
\end{proof}

%%%%%%%%%%%%%%%%%%%%%%%%%%%%%%%%%%%%%%%%%%subsec
\subsection{The Poisson Foliation}
The left--invariance of the Lie bracket $[\cdot, \cdot]$ and the assumed left--invariance of the commutative product $\circ$ collectively imply that the Leibnizator $L$ is a left--invariant trilinear map. By construction, the local vector fields spanning the Poisson--algebra distribution $\mathcal{D}_{\text{Pois}}$ are restricted to the space of left--invariant fields $\mathfrak{X}^+(G)$. Consequently, $\mathcal{D}_{\text{Pois}}$ constitutes a left--invariant subbundle of the tangent bundle $TG$. The fiber of $\mathcal{D}_{\text{Pois}}$ at any point $g \in G$ is uniquely determined by its evaluation at the identity $e \in G$ via the differential of the left translation:
$$(\mathcal{D}_{\text{Pois}})_g = (d l_g)_e (\mathcal{D}_{\text{Pois}})_e \quad \forall g \in G \; .$$ 
As a smooth left--invariant subbundle, $\mathcal{D}_{\text{Pois}}$ possesses constant rank globally, ensuring that the distribution is regular on $G$.\\

The regularity of $\mathcal{D}_{\text{Pois}}$, coupled with its involutivity (as established in Proposition \ref{involutive}), ensures its complete integrability by the Frobenius Theorem. The maximal integral manifold (leaf) passing through the identity $e$ integrates the Lie subalgebra $(\mathcal{D}_{\text{Pois}})_e \subseteq \mathfrak{g}$ to a unique connected immersed Lie subgroup $H \subseteq G$. Due to the left--invariance of the distribution, every other leaf of the foliation is obtained by the left translation of $H$. Thus, $\mathcal{D}_{\text{Pois}}$ induces a global foliation of $G$ whose leaves are the left cosets $\{gH\}_{g \in G}$.

\begin{theorem}\label{ProperDpois}
Let $G$ be an F--Lie group equipped with the left--invariant connection defined by $\nabla_{u^+} v^+ = \tfrac{1}{2}[u^+, v^+] + u^+ \circ v^+$. The Poisson-algebra distribution $\mathcal{D}_{\text{Pois}}$ satisfies the following geometric properties:
\begin{enumerate}
\item The distribution $\mathcal{D}_{\text{Pois}}$ is autoparallel with respect to $\nabla$. That is, for any pair of left--invariant vector fields $u^+, v^+$ taking values in $\mathcal{D}_{\text{Pois}}$, the covariant derivative $\nabla_{u^+} v^+$ also resides in $\mathcal{D}_{\text{Pois}}$.

\item The leaves of the foliation $\mathcal{F}_{\text{Pois}}$ induced by $\mathcal{D}_{\text{Pois}}$ are totally geodesic submanifolds of $G$. Equivalently, any geodesic of $G$ starting tangent to a leaf remains within that leaf for all time.

\item The holonomy algebra of the connection $\nabla$ restricted to the leave $H \subseteq G$ coincides with the specialized Poisson holonomy algebra $\mathfrak{hol}(\nabla^{\text{Pois}}\big|_H)$, characterized by:
$$\mathfrak{hol}(\nabla^{\text{Pois}}\big|_H) = \operatorname{ad}_{[\mathfrak{h}, \mathfrak{h}]} + A_{[\mathfrak{h}, [\mathfrak{h}, \mathfrak{h}]]} \, ,$$
where $\mathfrak h=(\mathcal{D}_{\text{Pois}})_e \subseteq \mathfrak{g}$ and $A$ denotes the endomorphism $A_u = \tfrac{1}{2}\operatorname{ad}_u + S_u$.

\end{enumerate}

\end{theorem}

\begin{proof}
\begin{enumerate}
\item Let $u, v \in (\mathcal{D}_{\text{Pois}})_e \subseteq \mathfrak{g}$ be elements of the Poisson subspace at the identity. By the algebraic closure properties established in Proposition  \ref{subalgebra} and Proposition \ref{involutive}, it follows that both  the commutative product $u \circ v$ and the Lie bracket $[u, v]$ reside in $(\mathcal{D}_{\text{Pois}})_e$. By the definition of the left--invariant connection, we have: $\nabla_{u^+} v^+ = \tfrac{1}{2}[u^+, v^+] + u^+ \circ v^+$. Evaluating this expression at any point $g \in G$, the left--invariance of the operations ensures that the resulting vector field $\nabla_{u^+} v^+$ remains a section of the left--invariant subbundle $\mathcal{D}_{\text{Pois}}$. Consequently, $\mathcal{D}_{\text{Pois}}$ is autoparallel with respect to $\nabla$.

\item In the context of affine geometry, the autoparallelism of a distribution $\mathcal{D}$ is equivalent to the condition that any geodesic $\gamma(t)$ of the connection $\overline \nabla$ whose initial tangent vector $\dot{\gamma}(0)$ belongs to $\mathcal{D}$ must remain tangent to the distribution for its entire duration. By the Frobenius Theorem, $\mathcal{D}_{\text{Pois}}$ integrates to a foliation whose leaves are, by definition, totally geodesic submanifolds of the affine manifold $(G, \nabla)$.

\item The characterization of the holonomy algebra $\mathfrak{hol}(\nabla|_{\mathcal{F}_{\text{Pois}}})$ follows directly from the restriction of the ambient connection to the integral manifolds of the distribution. By utilizing the inclusion established in Proposition \ref{thm:holonomy} and the structural analysis of the Leibnizator's kernel, the holonomy algebra is precisely identified as the specialized Poisson holonomy algebra $\operatorname{ad}_{[\mathfrak{h}, \mathfrak{h}]} + A_{[\mathfrak{h}, [\mathfrak{h}, \mathfrak{h}]]}$.

\end{enumerate}
 \end{proof}

\begin{rem}
 Let $G$ be an F--Lie group such that the Poisson-algebra distribution $\mathcal{D}_{\text{Pois}}$ is non-trivial (i.e., $\mathcal{D}_{\text{Pois}} \neq \{0\}$). Under this condition, the manifold $G$ admits a canonical foliation by connected, immersed submanifolds. Since these leaves are integral manifolds of a distribution characterized by the vanishing of the Leibnizator, they naturally inherit the structure of F--Poisson Lie groups.
\end{rem}

%%%%%%%%%%%%%%%%%%%%%%%%%%%%%%%%%%%%%%%%%%subsec
\subsection{The Heisenberg Lie Algebra} 
In this subsection, we conduct a specialized study of F$_{\text{man}}$--structures defined over the Heisenberg Lie algebra, which we denote by $\mathfrak{n}_{3,1}$ (following the convention in \cite{CLD}). Our objective is to explicitly compute the curvatures, holonomies, and Poisson--algebra distributions associated with the canonical F--connection in this specific algebraic setting.

The classification of F$_{\text{man}}$--algebras up to dimension three has been previously established \cite{CLD}. The four classified F$_{\text{man}}$--algebra structures on the three--dimensional Heisenberg Lie algebra $\mathfrak{n}_{3,1}$, with basis $\{e_1, e_2, e_3\}$, are summarized below (see Theorems 3.1, 3.3, 3.7, and 3.8 of \cite{CLD}). 

\begin{table}[h]
\centering
\begin{tabular}{|c|c|c|c|}
\hline
\textbf{Case} & \textbf{\small Non-Zero Lie Brackets} & \textbf{\small Product $\circ$} & \textbf{\small Non-Zero Products} \\
\hline
\text{(1a)} & $\mathfrak{n}_{3,1}^{(1)}:\;\;\;\;[e_2, e_3] = e_1$ & $A_3$ & $ e_2 \circ e_2 = e_1$,\; $e_3 \circ e_3 = e_1$  \\
\hline
\text{(1b)} & $\mathfrak{n}_{3,1}^{(1)}:\;\;\;\; [e_2, e_3] = e_1$ & $A_4$ & $e_2 \circ e_3 = e_1$,\; $e_3 \circ e_3 = e_2$ \\
\hline
\text{(2a)} & $\mathfrak{n}_{3,1}^{(2)}:\;\;\;\; [e_1, e_3] = e_2$ & $D_1$ & $e_1 \circ e_1 = e_2$,\; $e_3 \circ e_3 = e_3$ \\
\hline
\text{(2b)} & $\mathfrak{n}_{3,1}^{(2)}:\;\;\;\; [e_1, e_3] = e_2$ & $D_2$ & $e_1 \circ e_1 = e_2$, $e_1 \circ e_3 = e_1$,  $e_2 \circ e_3=e_2$, $e_3 \circ e_3=e_3$  \\
\hline
\end{tabular}
\caption{F$_{\text{man}}$--algebra structures on the Heisenberg Lie algebra ($\mathfrak{n}_{3,1}$)} \label{tab:fman_n31}
\end{table}

It can be directly verified that the four combinations of Lie brackets and commutative associative products listed in the Table \ref{tab:fman_n31} satisfy the Hertling--Manin condition, thereby defining an F$_{\text{man}}$--algebra structure. For subsequent analysis, we denote these four F$_{\text{man}}$--algebras as follows: $(\mathfrak{n}_{3,1}^{(1)}, A_3)$, $(\mathfrak{n}_{3,1}^{(1)}, A_4)$, $(\mathfrak{n}_{3,1}^{(2)}, D_1)$, and $(\mathfrak{n}_{3,1}^{(2)}, D_2)$.\\

The geometric properties of the F$_{\text{man}}$--algebras defined over the Heisenberg Lie algebra $\mathfrak{n}_{3,1}$ (as classified in Table \ref{tab:fman_n31}) are summarized in the following theorem. These properties relate to the Poisson condition ($L \equiv 0$), the flatness of the canonical connection $\nabla$ ($R \equiv 0$), and the structure of the Poisson--algebra distribution $\mathcal{D}_{\text{Pois}}$.
 
\begin{theorem}
The four F$_{\text{man}}$--algebra structures over the Heisenberg Lie algebra $\mathfrak{n}_{3,1}$ possess the following geometrical characteristics:

\begin{enumerate}
\item The F$_\text{man}$--algebra $(\mathfrak n_{3,1}^{(1)}, A_3)$:
\begin{itemize}

\item It constitutes a Poisson algebra, as the Leibnizator vanishes identically ($L \equiv 0$).

\item The associated canonical F--connection $\nabla$ is flat, with a vanishing curvature tensor ($R \equiv 0$).

\end{itemize}

\item The F$_\text{man}$--algebra $(\mathfrak n_{3,1}^{(1)}, A_4)$: 

\begin{itemize}
\item It is a non--Poisson algebra ($L \not\equiv 0$).

\item The associated Poisson--algebra distribution $\mathcal{D}_{\text{Pois}}$ is of rank two. 

\item The canonical F--connection $\nabla$ is flat ($R \equiv 0$).

\end{itemize}

\item The F$_\text{man}$--algebra $(\mathfrak n_{3,1}^{(2)}, D_1)$:

\begin{itemize}
\item It is a non--Poisson algebra ($L \not\equiv 0$).

\item The Poisson--algebra distribution $\mathcal{D}_{\text{Pois}}$ is of rank two.

\item The canonical F--connection $\nabla$ possesses non--vanishing curvature ($R \not\equiv 0$).

\item The holonomy Lie algebra $\mathfrak{hol}(\nabla)$ is abelian and of dimension one. 
\end{itemize}

\item The F$_\text{man}$--algebra $(\mathfrak n_{3,1}^{(2)}, D_2)$: 
\begin{itemize}
\item It is a non-Poisson algebra ($L \not\equiv 0$).

\item The Poisson--algebra distribution $\mathcal{D}_{\text{Pois}}$ is of rank two

\item The canonical F--connection $\nabla$ possesses non--vanishing curvature ($R \not\equiv 0$).

\item The holonomy Lie algebra $\mathfrak{hol}(\nabla)$ is abelian and one--dimensional.
\end{itemize}
\end{enumerate}
\end{theorem}

\begin{proof}
The results are established by explicitly computing the Leibnizator $L$ and the Riemann curvature tensor $R$ for each F$_\text{man}$--algebra structure, utilizing the decomposition derived in Theorem \ref{curvaturethm}:
$$R(u, v) w = R^0(u, v) w + R^L(u, v) w$$
where $R^0(u, v)w = -\tfrac{1}{4}[[u, v], w]$ denotes the intrinsic Lie curvature and $R^L(u, v)w = \tfrac{1}{2}\left(L(u, v, w) - L(v, u, w)\right)$ represents the contribution from the F--structure.\\

Let $\{e_1, e_2, e_3\}$ be the basis for $\mathfrak{n}_{3,1}$ (cf. Table \ref{tab:fman_n31}) and $\{\theta^1, \theta^2, \theta^3\}$ its corresponding dual basis, such that $\theta^i(e_j)=\delta^i_j$.

\begin{enumerate}
\item F$_{\text{man}}$--algebra $(\mathfrak{n}_{3,1}^{(1)}, A_3)$:

\begin{itemize}
\item \textbf{Leibnizator}: Direct computation shows $L(e_i, e_j, e_k) = 0$ for all indices, hence $L \equiv 0$. This confirms the structure is a Poisson algebra.

\item \textbf{Curvature}: Both $R^L$ and $R^0$ vanish identically (the latter due to the nilpotency of the Heisenberg algebra acting on itself, $ [[e_i, e_j], e_k] = 0$ for $i,j,k=1,2,3$). Thus, $R = 0$, and the connection is flat.

\end{itemize}

\item F$_\text{man}$--algebra $(\mathfrak n_{3,1}^{(1)}, A_4)$:

\begin{itemize}
\item \textbf{Leibnizator}: The non--vanishing components yield $L = -\theta^3 \otimes \theta^3 \otimes \theta^3 \otimes e_1$. Since $L \not\equiv 0$, it is a non--Poisson F$_\text{man}$--algebra.

\item \textbf{Poisson--algebra Distribution}: $(\mathcal{D}_{\text{Pois}})_e$ is the kernel of the map $u \mapsto L(u, \cdot, \cdot)$. Since $L(e_1, \cdot, \cdot)=0 = L(e_2, \cdot, \cdot)$, the subspace is spanned by $\{e_1, e_2\}$, which is of rank two.

\item \textbf{Curvature}: Both $R^0$ and $R^L$ vanish, resulting in a flat connection ($R = 0$).

\end{itemize}

\item F$_\text{man}$--algebra $(\mathfrak n_{3,1}^{(2)}, D_1)$:\

\begin{itemize}
\item \textbf{Leibnizator}: $L = \theta^1 \otimes \theta^3 \otimes \theta^3 \otimes e_2 \not\equiv 0$. Then, it is a non--Poisson F$_\text{man}$--algebra.

\item \textbf{Poisson--algebra Distribution}: $(\mathcal{D}_{\text{Pois}})_e = \operatorname{span}\{e_2, e_3\}$, having rank two.

\item \textbf{Curvature}: We find $R^0 = 0$, while the F--structure term yields:
$$R^L = \tfrac{1}{2}(\theta^1 \otimes \theta^3 - \theta^3 \otimes \theta^1) \otimes \theta^3 \otimes e_2 = (\theta^1 \wedge \theta^3) \otimes \theta^3 \otimes e_2 \; .$$
As $R = R^L \neq 0$, the structure is non--flat.

\item \textbf{Holonomy}: We characterize the connection $\nabla$ and the curvature $R$ as vector space endomorphisms taking values in differential 1--forms and 2--forms, respectively. Their explicit tensor expressions are given by:
\begin{align*}
R =& (\theta^1 \wedge \theta^3) \otimes \theta^3 \otimes e_2 \, , \\
\nabla=&\tfrac 12 \theta^1 \otimes (\theta^3 \otimes e_2)- \tfrac 12 \theta^3 \otimes (\theta^1 \otimes e_2)+\theta^1 \otimes (\theta^1 \otimes e_2) +\theta^3 \otimes (\theta^3 \otimes e_3) \, .
\end{align*}
The action of these operators on a general vector $v = v^1 e_1 + v^2 e_2 + v^3 e_3 \in \mathfrak{n}_{3,1}$ is computed as follows. The curvature operator $R$ acts on $v$ as an endomorphism-valued 2-form:
\begin{equation*}
R^L(v)=\left( \theta^1 \wedge \theta^3 \right) \otimes e_2 \;\; (\theta^3 (v))= \left( \theta^1 \wedge \theta^3 \right) \otimes (v^3 e_2) 
\end{equation*}
Similarly, the connection 1-form acts on $v$ as:
\begin{align*}
\nabla (v)=&\tfrac 12 \theta^1 \otimes (\theta^3(v) \;  e_2  )- \tfrac 12\theta^3 \otimes (\theta^1(v)\; e_2)  + \theta^1 \otimes (\theta^1(v) \; e_2 ) +\theta^3   \otimes (\theta^3(v)\; e_3) \, ,\\
=&\theta^1 \otimes \left(\frac{v^3}{2}-v^1 \right) e_2 + \theta^3  \otimes \left( v^3 e_3-\tfrac{v^1}{2} e_2 \right) 
\, .
\end{align*}
Applying the Ambrose--Singer Theorem, and noting that $[\nabla, R] = 0$, the holonomy Lie algebra $\mathfrak{hol}(\nabla)$ is spanned by the curvature endomorphisms and is therefore one--dimensional.

\end{itemize}

\item F$_\text{man}$--algebra $(\mathfrak n_{3,1}^{(2)}, D_2)$:
\begin{itemize}
\item \textbf{Leibnizator}: $L = -\theta^1 \otimes \theta^3 \otimes \theta^3 \otimes e_2 \not\equiv 0$. Then, it is a non--Poisson F$_\text{man}$--algebra.

\item \textbf{Poisson--algebra Distribution}: $(\mathcal{D}_{\text{Pois}})_e = \operatorname{span}\{e_2, e_3\}$, having rank two.

\item \textbf{Curvature}: We get $R^0 = 0$, while the $R^L$ term is given by:
$$R^L = \tfrac{1}{2}(\theta^3 \otimes \theta^1 - \theta^1 \otimes \theta^3) \otimes \theta^3 \otimes e_2 = -(\theta^1 \wedge \theta^3) \otimes \theta^3 \otimes e_2 \; .$$
As $R = R^L \neq 0$, the structure is non--flat.

\item \textbf{Holonomy}: The tensor expressions for the connection is:
\begin{align*}
\nabla=&\tfrac 12 \theta^1 \otimes (\theta^3 \otimes e_2)- \tfrac 12 \theta^3 \otimes (\theta^1 \otimes e_2)+\theta^1 \otimes (\theta^1 \otimes e_2) +\theta^1 \otimes (\theta^3 \otimes e_1)+ \theta^2 \otimes (\theta^3 \otimes e_2) \\
&+\theta^3 \otimes (\theta^1 \otimes e_1)+\theta^3 \otimes (\theta^2 \otimes e_2)+\theta^3 \otimes (\theta^3 \otimes e_3) \; .
\end{align*}
The holonomy Lie algebra $\mathfrak{hol}(\nabla)$ is one--dimensional because it is spanned by the curvature endomorphisms.
\end{itemize}

\end{enumerate}
\end{proof}

\begin{rem}\label{rem:holonomy_heis}
Since the Heisenberg Lie algebra $\mathfrak{n}_{3,1}$ is a two--step nilpotent Lie algebra, the triple Lie bracket vanishes. Consequently, the intrinsic Lie curvature term $R^0$ vanishes identically. The total curvature $R$ is therefore entirely determined by the Leibnizator $L$: $R(u, v) w = R^L(u, v) w = \tfrac{1}{2}\left(L(u, v, w) - L(v, u, w)\right)$.\\

In the two non--flat cases, $(\mathfrak{n}_{3,1}^{(2)}, D_1)$ and $(\mathfrak{n}_{3,1}^{(2)}, D_2)$, the curvature tensor $R(u, v)$ has a constant rank one as an endomorphism on $\mathfrak{n}_{3,1}$. The curvature can be compactly expressed in the form $R(u, v) = \Omega(u, v)\, J$, where $\Omega$ is a fixed non--zero 2-form, and $J \in \mathfrak{gl}(\mathfrak{n}_{3,1})$ is the endomorphism defined on the basis by $J(e_3) = e_2$ and $J(e_1) = 0= J(e_2)$.\\

The endomorphism $J$ satisfies the nilpotency condition $J^2 = 0$. The Lie algebra of the holonomy group $\mathfrak{hol}(\nabla)$ is one--dimensional and spanned by $J$. Consequently, the associated holonomy group consists of unipotent transformations of the form $\exp(tJ) = I + tJ$, which represent infinitesimal displacements along the fixed direction $e_2$.\\

Geometrically, this reveals that although the connection $\nabla$ is not flat ($R \not\equiv 0$), its curvature only generates parallel transport effects in a single, fixed direction ($e_2$). This results in a minimal (rank--one) nontrivial holonomy.

\end{rem}

%%%%%%%%%%%%%%%%%%%%%%%%%%%%%%%%%%%%%%%%%%sec

\section*{Acknowledgements}
%The authors thank the anonymous referees for their valuable suggestions and corrections, which improved the quality and readability of this article.\\

S. Casta\~neda-Montoya gratefully acknowledges the financial support provided by FaMAF--UNC and CIEM--CONICET. He also expresses sincere gratitude to the Institute of Mathematics at the Universidad de Antioquia (Medell\'in, Colombia) for their hospitality and support during his research visit, spanning the period from March 2025 to March 2026, where the completion of this project was achieved.
%%%%%%%%%%%%%%%%%%%%%%%%%%%%%%%%%%%%%%%

\section*{Conflict of Interest}
The authors declare that they have no conflicts of interest relevant to the content of this article.

%%%%%%%%%%%%%%%%%%%%%%%%%%%%%%%%%%%%%%%Bibl

\end{document}